\documentclass[12pt]{article}
\usepackage[russian,english]{babel}
\usepackage{amsmath,amsfonts,amssymb,amsthm,graphics}
\usepackage{euscript,epsfig}

\newtheorem{thm}{Theorem}[section]
\newtheorem{lem}[thm]{Lemma}
\newtheorem{cor}[thm]{Corollary}
\newtheorem{con}{Conjecture}
\theoremstyle{definition}
\newtheorem{defin}[thm]{Definition}
\newtheorem*{rem}{Remark}
\newtheorem{exam}[thm]{Example}
\newtheorem*{notat}{Notation}
\newcommand { \ib }[1] {\textit{\textbf{#1}}}

\newcommand { \lin }{\mathop{\rm{lin}}\nolimits}

\begin{document}
\renewcommand{\ib}{\mathbf}
\renewcommand{\proofname}{Proof}
\makeatletter \headsep 10 mm \footskip 10 mm
\renewcommand{\@evenhead}%

\title{Belt distance between facets of space-filling zonotopes\footnote{This work is supported by RFBR grant 08-01-0054652-a and by grant ``Scientific Schools'' \Russian НШ-5413.2010.1 \English}.}
\author{A. Garber, Moscow State University, Russia}
\maketitle
\begin{abstract}
For every $d$-dimensional polytope $P$ with centrally symmetric facets we can associate a ``subway map'' such that every line of this ``subway'' corresponds to set of facets parallel to one of ridges $P.$ The belt diameter of $P$ is the maximal number of line changes that you need to do in order to get from one station to another.

In this paper we prove that belt diameter of $d$-dimensional space-filling zonotope is not greater than $\lceil \log_2\frac45d\rceil$. Moreover we show that this bound can not be improved in dimensions $d\leq 6.$
\end{abstract}
\section{Parallelohedra and Voronoi's conjecture.}

\begin{defin}
A polytope $P\subset\mathbb{R}^d$ is called a {\it $d$-dimensional parallelohedron} if we can tile $\mathbb{R}^d$ by its parallel copies. For example a cube $C^d=[0;1]^d$ is a parallelohedron.
\end{defin}

\begin{defin}
The {\it Dirichlet-Voronoi domain} or {\it Dirichlet-Voronoi polytope} for a $d$-dimensional lattice $\Lambda^d$ in $\mathbb{R}^d$ is called a polytope consist of all points of $\mathbb{R}^d$ which are closer to a fixed lattice point $O\in\Lambda^d$ than to any other point from $\Lambda^d,$ i.e.
$$DV_{\Lambda^d}:=\{X\in\mathbb{R}^d:\text{for any }
Y\in \Lambda^d \text{ holds }XO\leq XY, O\in\Lambda^d\ \text{ is fixed}\}.$$
\end{defin}

It is clear that the Dirichlet-Vornoi domain does not depend on the point $O$ and it depends only on the lattice $\Lambda^d.$ So it is evident that the Dirichlet-Voronoi polytope for an arbitrary lattice is a parallelohedron. One of the main conjectures in the theory of parallelohedra is a Voronoi's conjecture which states the converse.

\begin{con}[G.Voronoi, \cite{Vor}]\label{Voronoi}
Any $d$-dimensional polytope is affine equivalent to a Dirichlet-Voronoi polytope of some $d$-dimensional lattice.
\end{con}

In the following text we will use some theorems about parallelohedra.

\begin{thm}[H.Minkowski, \cite{Min}]\label{Minkowski}
Any $d$-dimensional parallelohedron satisfies the following properties:

\begin{enumerate}\renewcommand{\theenumi}{{\rm \arabic{enumi}}}
\item $P$ is centrally-symmetric;
\item any facet of $P$ is centrally-symmetric;
\item a projection of $P$ along any of its ridge (face of codimension $2$) is a parallelogram or a centrally-symmetric hexagon.
\end{enumerate}
\end{thm}

\begin{rem}
In this theorem as well as later if we say about ``projection along $k$-dimensional subspace $\pi$ of $d$-dimensional space'' then we consider a projection onto any $(d-k)$-dimensional space transversal to $\pi.$ All projections onto different transversal planes are affine equivalent and in the case of polytopes or sets of vectors all combinatorial properties are affine invariants.
\end{rem}

\begin{thm}[B.Venkov, \cite{Ven}]\label{Venkov}
If a polytope $P$ satisfies three conditions of the Minkowski theorem then $P$ is a parallelohedron.
\end{thm}

\begin{rem}
The first condition of the theorems by Minkowski and Venkov is redundant because due Alexandrov-Shephard theorem \cite{ASh} any polytope with centrally-symmetric facets is centrally-symmetric itself.
\end{rem}

\begin{defin}
Let $F$ be a ridge of $d$-dimensional parallelohedron $P.$ The set $\mathcal{B}$ of all facets of $P$ parallel to $F$ is called a {\it belt} of $P$ corresponding to $F.$

Due to the third condition of the Minkowski theorem any belt of parallelohedron consist of four or six facets.
\end{defin}

\begin{defin}
We say that the set $F_0,\mathcal{B}_1,F_1,\mathcal{B}_2,\ldots,\mathcal{B}_n,F_n$ of the faces $F_i$ and the belts $\mathcal{B}_i$ is a {\it belt path} on the parallelohedron $P$ if any facet $F_i$ lies in the belts $\mathcal{B}_i$ and $\mathcal{B}_{i+1}$ and any belt $\mathcal{B}_i$ contains the facets $F_{i-1}$ and $F_i.$ We say that {\it belt distance} between facets $F$ and $G$ of parallelohedron $P$ is equal to $k$ if for a belt path $\Gamma$ with $F=F_0$ and $F_n=G$ and with the smallest possible $n$ we have $n=k.$ If facets $P$ and $Q$ are opposite then we will treat a belt distance between $F$ and $G$ as 0.

The belt distance between facets $F$ and $G$ in the polytope $P$ we will denote $d_\mathcal{B}^P(F,G).$
\end{defin}

\begin{rem}
The belt distance defined as above is the same as a usual distance on a graph, whose vertices are all the facets of the parallelohedron $P$ and two vertices are connected by an edge if and only if the correspondent facets lies in some belt.
If we identify the vertices correspondent to the opposite facets then we will obtain a {\it Venkov graph} of the parallelohedron $P$ (see \cite{DG} and \cite{Ord} for more detailed definition and examples).
\end{rem}

In the same way we can define a belt distance between facets of any polytope with centrally symmetric facets.

\begin{defin}
The maximal belt distance between two facets of parallelohedron $P$ is called a {\it belt diameter} of $P.$
\end{defin}

For the moment it was proved some cases of the Voronoi conjecture. In particularly, in 1908 Voronoi proved that the conjecture \ref{Voronoi} holds if the parallelohedron $P$ is primitive, i.e. in any vertex of the tiling of $\mathbb{R}^d$ into parallel copies of $P$ only $d+1$ copies of $P$ meets together. Later, in 1929 Zhitomirskii \cite{Zhit} proved the Voronoi conjecture in the case if the parallelohedron is primitive in any of its ridge, i.e. in any ridge of the correspondent tiling exactly 3 copies of $P$ meets together, or equivalently, any belt of $P$ consists of exactly 6 facets. In 1999 Erdahl proved the Voronoi conjecture for space-filling zonotopes \cite{Erd}. In 2005 Ordine proved the Voronoi conjecture for the 3-irreducible parallelohedrons \cite{Ord}.

All the results mentioned above were obtain using a canonical scaling method. The existence of canonical scaling for a tiling of $\mathbb{R}^d$ into copies of $P$ is equivalent to the Voronoi conjecture for the parallelohedron $P$ (see \cite{Vor}).
{\it Canonical scaling} is a function which associates to any facet $F$ of the tiling of $\mathbb{R}^d$ some real number $n(F)$ such that for any ridge $G$ of the tiling there exist a collection of signs plus or minus with $$\sum_{i=1}^k\pm \ib e_i
n(F_i)=\ib 0.$$ Here $F_i,i=1,\ldots,k$ is the set of all facets of tiling meets at $G$ and $\ib e_i$ is a unit normal vector to $F_i.$ It is evident that $k=3$ or $k=4$ depending on the length of the belt generated by $G.$

If two facets $F$ and $G$ of tiling has a common ridge from the belt of length 6 then the value of the canonical scaling on the one of these facets is uniquely defined by the value on the another and vice versa. Therefore if we have a parallelohedron with relatively small belt diameter then we need to show the existence of the canonical scaling for $P$ on relatively small belt cycles. In this work we will show that the belt diameter of parallelohedral zonotope of dimension $d$ is not greater than $\lceil \log_2{\frac45d}\rceil$.

\section{Zonotopes.}

Most part lemmas of this section can be found in papers \cite{Shep} and \cite{Mcm}.

\begin{defin}\label{zonotope}
A poltyope $P\subset\mathbb{R}^d$ is called a {\it zonotope} if it is a projection of a cube $C^n$ of some dimension $n.$

Equivalently a polytope $P$ is a zonotope if it is a Minkowski sum of a finite set of segments $[\ib 0,\ib v_i],i=1,\ldots,n,$ for some vector set $V=\{\ib v_1,\ldots,\ib
v_n\}\in\mathbb{R}^{d\times n}$ of $n$ vectors from $\mathbb{R}^d.$ In this case a zonotope $P$ is denoted as $P=Z(V)=Z(\ib v_1,\ldots,\ib v_n).$
\end{defin}

\begin{rem}
In most papers, for example in the book of G\"{u}nter Ziegler ``Lectures on Polytopes'' \cite[Sect. 7.3]{Zig} a zonotope $Z(V)=Z(\ib v_1,\ldots,\ib v_n)$ is defined as a Minkowski sum of segments $[-\ib v_i,\ib v_i],i=1,\ldots,n.$ This definition is equivalent to our definition \ref{zonotope} because the resulting polytopes are homothetic with the coefficient 2.
\end{rem}

In two-dimensional case all zonotopes are centrally symmetric polygons and vice versa. In that case a polygon is a Minkowsi sum of segments correspondent to its sides.

There exist polytopes which are zonotopes and parallelohedra simultaneously. For example a cube $C^d$ is a parallelohedron (and a Dirichlet-Voronoi polytope for a standard cubic lattice $\mathbb{Z}^d$) and a zonotope. It is clear that the combinatorial diameter of a cube is equal to 1 because any two non-opposite facets of a cube has a common ridge.

The following example illustrate another polytope from both families of zonotopes and parallelohedra.

\begin{exam}[Permutahedron]
A {\it permutahedron} $\Pi_d$ is a convex hull of $(d+1)!$ points $\sigma(1,2,\ldots,d+1)$ for all permutations $\sigma$ from the group $\mathcal{S}_{d+1}.$ The polytope $\Pi_d$ lies in the $(d+1)$-dimensional space $\mathbb{R}^{d+1}$ but it is a $d$-dimensional polytope because all its vertices lies in a $d$-plane $x_1+\ldots+x_{d+1}=1+\ldots+(d+1).$

The permutahedron $\Pi_d$ is a zonotope. It can be represented as a Minkowski sum of $\displaystyle\frac{d(d+1)}{2}$ segments defined by vectors $\ib e_i-\ib e_j, 1\leq i<j\leq d+1,$ here $\ib e_k$ is a $k$-th vector of the standard basis in $\mathbb{R}^{d+1}.$

The polytope $\Pi_d$ is a parallelohedron and moreover it is the Dirichlet-Voronoi polytope for a lattice generated by vectors $\ib e_1+\ib e_2+\ldots+\ib e_{d+1}-(d+1)\ib e_k$ and these vectors are parallel to the hyperplane of the permutahedron $\Pi_d$ \cite{GP}.

Now we will find the belt diameter of the permutahedron.

There exists the following combinatorial description of all faces of the permutahedron \cite{Zig,GP}. For every face $F$ of codimension $k$ of $\Pi_d$ there exist an ordered partition of the set $\{1,2,\ldots, d+1\}$ into $k+1$ subsets $A_1,A_2,\ldots,A_{k+1}.$ In that case the vertices of $F$ are the points with the coordinates from 1 to $|A_1|$ on the places correspondent to elements from $A_1;$ coordinates from $|A_1|+1$ to $|A_1|+|A_2|$ on the places correspondent to elements from $A_2$ and so on (here $|X|$ denotes the cardinality of the set $X$). We will denote a face correspondent to a partition $A_1,\ldots, A_{k+1}$ as $F(A_1,\ldots,A_{k+1}).$

With such description a face $F(A_1,\ldots, A_k)$ contains in a face $F(B_1,\ldots,B_m)$ with $k\geq m$ if and only if $B_1=A_1\cup\ldots\cup A_{k_1}, B_2=A_{k_1+1}\cup\ldots\cup A_{k_2}$ and so on for some parameters $k_1<k_2<\ldots<k_m=k.$
Therefore the correspondent to a ridge $F(X,Y,Z)$ contains exactly 6 facets
$$F(X\cup Y,Z),F(X,Y\cup Z),F(X\cup Z,Y),F(Z,X\cup Y),F(Y\cup Z,X)
\text{  }F(Y,X\cup Z).$$
And the neighbor facets in this list intersects in ridges
$$F(X,Y,Z), F(X,Z,Y), F(Z,X,Y),F(Z,Y,X),F(Y,Z,X)\text{  }F(Y,X,Z).$$

Consider two arbitrary facets $F(A,B)$ and $F(C,D)$ of the permutahedron $\Pi_d.$ Consider four subsets $A\cap C, A\cap D, B\cap C, B\cap D.$ If one of these subsets is empty then the combinatorial distance between $F(A,B)$ and $F(C,D)$ is equal to 1. Otherwise consider two belts correspondent to ridges $F(A\cap C,
A\cap D, B)$ and $F(C, A\cap D, B\cap D).$ The first one contains the facet $F(A,B)$ and the second one contains the facet $F(C,D).$ Also both these belts contains the facet $F((A\cap C)\cup B,A\cap
D)=F((B\cap D)\cup C,A\cap D).$ Therefore the belt diameter of $d$-dimensional permutahedron ($d\geq3$) is equal to 2.
\end{exam}

\begin{exam}
Consider two- and three-dimensional cases more accurately. On the plane there exists only two combinatorial types of parallelohedra, i.e. parallelogram and centrally symmetric hexagon. In both cases all ridges (i.e. vertices) lies in one belt, so the combinatorial diameter of any two-dimensional parallelohedron is equal to 1.

In three-dimensional case there exist five combinatorial distinct parallelohedra, they were obtain by Fedorov \cite{Fed}. The Fedorov's solids are cube, centrally symmetric hexagonal prism, rhombic dodecahedron, elongated dodecahedron and truncated octahedron.

\begin{figure}[!ht]
\begin{center}
\includegraphics[scale=0.6]{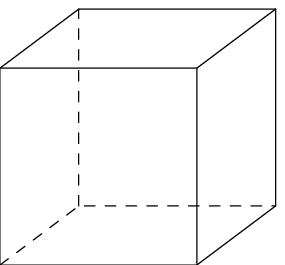}
\hskip 2cm
\includegraphics[scale=0.6]{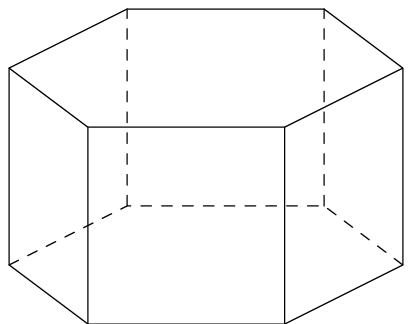}
\caption{Cube and hexagonal prism.}
\end{center}
\end{figure}

\begin{figure}[!ht]
\begin{center}
\includegraphics[scale=0.6]{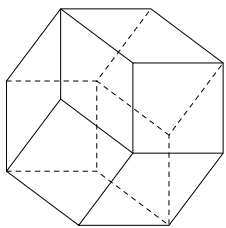}
\caption{Rhombic dodecahedron.}
\end{center}
\end{figure}

\begin{figure}[!ht]
\begin{center}
\includegraphics[scale=0.6]{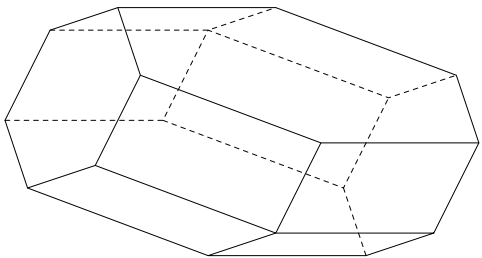}
\caption{Elongated dodecahedron.}
\end{center}
\end{figure}

\begin{figure}[!ht]
\begin{center}
\includegraphics[scale=0.6]{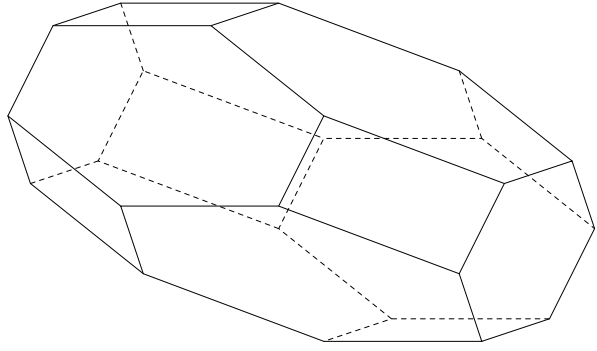}
\caption{Truncated octahedron.}
\end{center}
\end{figure}

It is easy to see that combinatorial diameters of a cube and a hexagonal prism are equal to 1 and combinatorial diameters of all other three-dimensional parallelohedra are equal to 2.

In two- and three-dimensional cases any parallelohedron is also a zonotope but is is not true even for dimension 4. The-well know 24-cell \cite[Tab. II]{24cell} is a parallelohedron but not a zonotope.
\end{exam}

Let us remind the definitions of the supporting plane and the face of polytope.

\begin{defin}
A hyperplane $\pi\subseteq\mathbb{R}^d$ is called a {\it supporting hyperplane} of the polytope $P\subseteq\mathbb{R}^d$ if $P$ lies in one of the closed halfspaces determined by the plane $\pi$ and the intersection $\pi \cap P$ is non-empty. In that case the intersection $\pi \cap P$ is called a {\it face} of $P$ and it is denoted as $F^P_\pi.$
\end{defin}

\begin{lem}\label{faces}
Consider $\pi$ is a supporting plane of the zonotope $Z(V).$ Let $V_\pi$ denotes a set of vectors consist of vectors of $V$ parallel to $\pi.$ Then polytopes $Z(V_\pi)$ and $F_\pi^{Z(V)}$ are parallel copies.
\end{lem}
\begin{rem}
If a zonotope has an empty set of generating vectors then we will treat it as one point (the origin).
\end{rem}

\begin{proof}
Let $Z(V)=Z(\ib v_1,\ldots \ib v_n)$ is a projection of the cube $C^n\subset\mathbb{R}^n$ onto the space $\mathbb{R}^d$ along the $(n-d)$-dimensional space $\pi.$ Consider a hyperplane $\pi\times\psi$ in the space $\mathbb{R}^n,$ this plane is a supporting plane for the cube $C^n$ and therefore it determines a face $F$ of the cube $C^n.$ The face $F$ is cube of some dimension and this cube is generated by sides of the cube $C^n$ which are projected to vectors of $V$ parallel to $\pi.$ Hence the whole face $F$ is projected into the parallel copy of the zonotope $Z(V_\pi)$ and also this face is projected into the face $F^{Z(V)}_\pi.$
\end{proof}

\begin{cor}\label{faces2}
Let $V$ be a set of vectors from $\mathbb{R}^d$ and $U$ is a proper subset of $V$ such that $(\lin U)\cap V=U$ (here $\lin U$ denotes the linear span of the vector-set $U$). Then the set $U$ determines some set of parallel faces of the zonotope $Z(V)$ and these faces are translates of the zonotope $Z(U)$ and has dimension $\dim\lin U.$ If $\dim\lin U=d-1$ then $U$ determines a pair of opposite facets of $Z(V).$
\end{cor}

\begin{proof}
It is enough to take a supporting hyperplane $\pi$ parallel to any vector from $U$ and not parallel to any vector from $V\setminus U$ and to apply the lemma \ref{faces} to the hyperplane $\pi.$
\end{proof}

\begin{lem}\label{2or3}
A zonotope $Z(V)$ of dimension $d$ is a parallelohedron if and only if for any $(d-2)$-dimensional vector-subset $U\subset V$ the projections of vectors from $V$ along subspace $\lin U$ gives vectors of at most three directions.\end{lem}

\begin{proof}
By the corollary \ref{faces2} all ridges of the zonotope $Z(V)$ can be determined by all $(d-2)$-dimensional subsets $U$ of the vector-set $V$ and a projection of $Z(V)$ along the ridge determined by the single subset $U$ is a zonotope constructed on the projections of the vectors from $V$ along the subset $U$. Here vectors from $\lin U$ are projected into zero-vector and they does not change Minkowski sum of other vectors.

Now it is sufficient to apply theorems of Minokowski \ref{Minkowski} and Venkov \ref{Venkov} to the polytope $Z(V)$ and to mention that two-dimensional zonotope $Z(K)$ is a parallelogram or a centrally symmetric hexagon if and only if the set $K$ contains vectors of at most three directions.
\end{proof}

\begin{lem}\label{setminus}
Consider the zonotope $Z(V\cup\{\ib u\})=Z(\ib v_1,\ib v_2,\ldots,\ib v_n, \ib u)\subseteq{R}^d$ is a $d$-dimensional parallelohedron and the zonotope $Z(V)=Z(\ib v_1,\ib v_2,\ldots,\ib v_n)$ is a $d$-dimensional zonotope then $Z(V)$ is a $d$-dimensional paralelohedron.
\end{lem}

\begin{proof}
Let $K$ is a $(d-2)$-dimensional vector-subset of the set $V,$ then $K$ is a $(d-2)$-dimensional subset of the set $V\cup\{\ib u\}.$ By lemma \ref{2or3} vectors of projections of $V\cup\{\ib u\}$ along $K$ are of at most three directions therefore the same statement is true for the set $V.$ So the zonotope $Z(V)$ satisfies conditions of the lemma \ref{2or3}, q.e.d.
\end{proof}

In the same way we can prove the following lemma.

\begin{lem}[A. Magazinov, \cite{Mag}]\label{alpha}
If the zonotope $Z(V)=Z(\ib v_1,\ldots, \ib v_{n-1}, \ib v_n)$ is a parallelohedron then the zonotope $Z(V')=Z(\ib v_1,\ldots,\ib v_{n-1}, \alpha \ib v_n)$ fir any $\alpha\neq 0$ is a parallelohedron too.
\end{lem}

The next lemma and corollary in more general case for arbitrary parallelohedra can be found in the following work of B.Venkov \cite{Ven2}.

\begin{lem}\label{projection}
Consider a space-filling zonotope $Z(V)$ and its edge $e.$ The projection of $Z(V)$ along $e$ is also a parallelohedron.
\end{lem}

\begin{proof}
According to the corollary \ref{faces2} there exists a vector $\ib v_i\in V$ which is a translation of an edge $e.$ Let $U$ be a set of projections of vectors from $V$ along vector $\ib v_i$ then projection $Z(V)$ along $e$ is a zonotope $Z(U).$

Let $K$ be a $(d-3)$-dimensional subset of $U$ and let $K$ determines a face $Z(K)$ of zonotope $Z(U).$ Denote as $K_V$ a subset of $V$ of vectors projected into vectors from $K\cup \{\ib 0\}$. First we will show that $K_V$ is $(d-2)$-dimensional set. After projection along vector $\ib v_i\in K_V$ the set $K_V$ projects on $(d-3)$-dimensional set $K$ and after this projection its dimension can decrease at most at 1. Moreover we have the following equality for linear spaces $\lin K_V=\lin K\oplus \lin\{\ib v_i\},$ so projection along $K_V$ is composition of projections along $K$ and $\ib v_i.$ Hence the set $K_V$ determines a $(d-2)$-dimensional face $Z(K_V)$ of parallelohedron $Z(V).$

Therefore the set of projections of vectors from $U$ along $K$ coincides with the set of projections of vectors from the set $V$ along $K_V,$ so vectors of $U$ are projected into vectors of at most three directions and zonotope $Z(U)$ satisfies conditions of the lemma \ref{2or3} and $Z(U)$ is a parallelohedron.
\end{proof}

\begin{cor}
Projection of space-filling zonotope $Z(P)$ along any of its face is a parallelohedron.\end{cor}

If set $V$ determines a zonotope $Z(V)$ and if we will remove a vector $\ib v$ from the set $V$ and $\dim(V\setminus\{\ib v\})<\dim V$ then the zonotope $Z(V\setminus\{\ib v\})$ is a projection of $Z(V)$ along the vector $\ib v.$ Therefore from lemmas \ref{setminus} and \ref{projection} we obtain the following corollary.

\begin{cor}
If zonotope $Z(V)$ is a parallelohedron and $U\subseteq V$ then $Z(U)$ is a parallelohedron too.\end{cor}

\begin{lem}
Let zonotope $Z(U)$ is a projection of zonotope $Z(V)$ along some face $Z(W)$ for some subset $W$ of $V.$ For every face $F_U$ of codimension $k$ of the zonotope $Z(U)$ there exists a unique face $F_V$ of codimension $k$ of the zonotope $Z(V)$ such that projection of $F_V$ along $Z(W)$ is $F_U.$ Moreover is a face $G_U$ is incident to $F_U$ then $G_V$ is also incident to $F_V.$
\end{lem}

\begin{proof}
The face $F_U$ is a zonotope $Z(U_1)$ for some subset $U_1\subseteq U.$ Consider a supporting plane $\pi$ correspondent to face $F_U$ of the zonotope $Z(U).$ For this plane we can find a supporting plane $\pi'$ of the zonotope $Z(V)$ such that $\pi$ is generated by plane $\pi$ and vectors of the set $W.$ A vector from $V$ is parallel to $\pi'$ if and only if it is in $W$ or it is projected into vector of the set $U_1.$ The face of $Z(V)$ determined by the hyperplane $\pi'$ is a face of codimension $k$ and its projection is the face $Z(U_1)=F_U,$ so the desired face $F_V$ has been constructed. The uniqueness of $F_V$ follows from the construction.

The preserving of the incidence relation with the construction of face $F_V$ and $G_V$ as described above is also evident because faces $F_U=Z(U_1)$ and $G_U=Z(U_2)$ are incident if and only if one of subsets $U_1$ and $U_2$ is a subset of another.
\end{proof}

\begin{cor}\label{projectedpath}
Let zonotope $Z(U)$ is a projection of zonotope $Z(V)$ along face $Z(W).$ For every belt path $\Gamma_U$ in zonotope $Z(U)$ there is a belt path $\Gamma_V$ in $Z(V)$ of the same length. If $\Gamma_U$ connects facets $F_U$ and $G_U$ of the zonotope $Z(U)$ then $\Gamma_V$ connects facets $F_V$ and $G_V$ of the zonotope $Z(V).$
\end{cor}

\section{Conjugated zonotopes.}

\begin{defin}
Let $E=\{\ib e_1,\ldots,\ib e_{d-1}\}$ and $F=\{\ib f_1,\ldots,\ib f_{d-1}\}$ be two sets of vectors in $\mathbb{R}^d.$ We will say that sets $E$ and $F$ are {\it conjugated} if for every $1\leq i\leq d-1$ we have the following equality for dimensions $\dim(E\cup\{\ib f_i\})=\dim (F\cup\{\ib e_i\})=d.$ The correspondent zonotope $Z(E\cup F)$ we will also call a {\it conjugated zonotope}.
\end{defin}

\begin{thm}\label{goodbasis}
Let $E$ and $F$ be two conjugated vector sets in $\mathbb{R}^d, d>2$ such that zonotope $Z(E\cup F)$ is a parallelohedron. There exists a space-filling zonotope $Z(V)$ combinatorially equivalent to the zonotope $Z(E\cup F)$ such that matrix $V$ with columns consist of coordinates of vectors from $E\cup F$ is the following:
$$V=\left(\begin{array}{c|c}
E_{d-1}&A\\\hline
0 \ldots 0& 1 \ldots 1
\end{array}
\right),$$
Here $E_{d-1}$ is unit $(d-1)\times(d-1)$-matrix and $A$ is a $0/1$-matrix of a size $(d-1)\times(d-1),$ i.e. entries of the matrix $A$ are zeros ore ones, moreover we can consider that at least half of entries in each row of the matrix $A$ are zeros.
\end{thm}

\begin{proof}
Note that the following transformation of the vector set does not change combinatorial type of the correspondent zonotope and does not change its property to be or to be a parallelohedron:

\begin{itemize}
\item non degenerated affine transformation of the vector set;

\item change of the basis of the comprehensive space (actually this transformation changes only coordinates of vectors but does not change vectors of the set);

\item multiplication one of vectors on non-zero constant.
\end{itemize}

The first two transformations are evidently satisfies the mentioned property. The third transformation satisfies the second part of the property because of lemma \ref{alpha} and the first part of the property because of lemma \ref{faces2}. Also it is clear that after any of described transformations sets $E$ and $F$ remains conjugated.

We will apply several transformations to sets $E$ and $F$ in order to obtain a set $V$ with the desired property; vectors obtained after some of these transformations we will denote also $\ib e_1,\ldots,\ib e_{d-1},\ib f_1,\ldots,\ib f_{d-1}$ as vectors of initial sets $E$ and $F.$ First we will consider a basis $\ib g_1,\ldots,\ib g_d$ such that first $d-1$ its vectors are vectors of the set $E$ and the last vector is some vector $\ib e_d.$ In this basis any vector $\ib f_i$ has non-zero last coordinate because dimension of the vector set $E\cup \{\ib f_i\}$ is equal to $d.$ Let multiply every vector $\ib f_i$ on a constant $\alpha_i$in order to make the last coordinate of the new vector $\ib f_i$ equal to $1.$

Let's apply lemma \ref{2or3} to the space-filling zonotope $Z(E\cup F)$ and its ridge generated by the vector set $E\setminus\{\ib e_i\};$ here we will consider a projection on the plane generated by vectors $\ib e_i$ and $\ib e_d.$ The vector $\ib f_j$ could not be projected in the vector collinear to $\ib e_i$ because in that case vector set $E\cup\{\ib f_j\}$ has dimension $d-1$ and this contradict to conjugacy of vector sets $E$ and $F.$ Therefore vectors of the set $F$ projects on vectors of one or two directions. Assume that vector $\ib f_j$ projects on vector $c_j\ib e_i+\ib e_d$ and $c_j$ here is a $j$-th coordinate of the vector $\ib f_j.$ Two vector of projections of $\ib f_j$ and $\ib f_k$ are collinear if and only if $c_j=c_k$ so between all $c_j$'s there are at most two different. All $c_j$'s could not be equal because in that case all vector from $F$ lies in the hyperplane $c_jx_d=x_i$ so as all vectors from $E$ except $\ib e_i$ and this statement contradicts with the conjugacy of $E$ and $F$ if $d>2.$ Denote the pair of values for $i$-th coordinate of vectors $\ib f_j$ as $a_i$ and $b_i$ with $a_i<b_i.$

We will multiply every vector $\ib e_i$ on the number $b_i-a_i$ and consider all vectors from $E\cup F$ in the new basis $\ib g_i'=(b_i-a_i)\ib g_i$ if $1\leq i\leq d-1$ and $\ib g_d'=g_d+a_1\ib g_1+\ldots+a_{d-1}\ib g_{d-1}.$ In this basis new vectors from $E$ are the first $d-1$ basic vectors and every vector from $F$ has coordinates $0$ or $1$ because it is a sum of $\ib g_d'$ with some vectors $\ib g_i'$ (namely vectors correspondent to $b_i$'s in the representation of the $\ib f_i$ in the old basis). Hence in the basis $\ib g_1',\ldots, \ib g_d'$ the vector set $E\cup F$ is represented by matrix
$$E\cup F=\left(\begin{array}{c|c}
E_{d-1}&A'\\\hline
0 \ldots 0& 1 \ldots 1
\end{array}
\right)$$
for a $0/1$-matrix $A'.$

Assume that there are less than one half of zeros in the $i$-th row of the matrix $A'.$ Then we will multiply the vector $\ib e_i$ on $-1$ and change the vector $\ib g_i'$ on the vector $\ib g_i''=-\ib g_i'$ and the vector $\ib g_d'$ on the vector $\ib g_d''=\ib g_d'+\ib g_i'.$ After applying this operation for every $i$ we will obtain the desired matrix for the set $E\cup F.$
\end{proof}

\section{Main results.}

\begin{thm}\label{conjugate}
Let $P$ and $Q$ be two facets of $d$-dimensional zonotope $Z(V).$ There exists a conjugated zonotope $Z(E\cup F)$ of dimension at most $d$ such that the belt distance between $P$ and $Q$ in $Z(V)$ is not greater than the belt distance between facets determined by conjugated sets $E$ and $F$ in $Z(E\cup F).$
\end{thm}

\begin{proof}
We will prove the statement by induction; the basis of induction $d=2$ is evident. Assume that the theorem is true for every dimension less than $d.$

If there is a vector $\ib v\in V$ such that both facets $P$ and $Q$ has edges parallel to $\ib v$ then project the zonotope $Z(V)$ along the edge $\ib v.$ Let $P'$ and $Q'$ are facets of projection $Z(V')$ correspondent to $P$ and $Q.$ By the corollary \ref{projectedpath} we have the inequality $d_\mathcal{B}^{Z(V)}(P,Q)\leq d_\mathcal{B}^{Z(V')}(P',Q').$ So it is sufficient to apply the induction assumption to the zonotope $Z(V')$ and facets $P'$ and $Q'.$

Consider that $P$ and $Q$ does not have parallel edges. The facet $P$ is a $(d-1)$-dimensional zonotope $P=Z(U)$ for some $(d-1)$-dimensional vector set $U\subset V.$ Let's choose an arbitrary basis $E$ in the set $U$ consist of vectors of the set $V.$ In the same way we can construct $(d-1)$-dimensional linearly independent set $F\subset V$ correspondent to the facet $Q.$ We will show that sets $E$ and $F$ are conjugate; it is enough to show that an arbitrary vector $\ib f_i\in F$ is linearly independent with vectors from $E.$ If it is not so then $\ib f_i$ is parallel to hyperplanes of both facets $P$ and $Q$ and this case is already done before.

We will show that zonotope $Z(E\cup F)$ is the zonotope from the statement of the theorem. For every belt path $\Gamma$ connecting facets $Z(E)$ and $Z(F)$ in $Z(E\cup F)$ we will construct a belt path of the same length connecting facets $P$ and $Q$ in $Z(V).$ Consider $\Gamma$ contains facets $Z(E)=Z(U_0),Z(U_1),\ldots,Z(U_n)=Z(F)$ with $U_i\subset E\cup F$ and $\dim(U_i\cap U_{i+1})=d-2$ because adjacent facets in belt path lies in one belt. Denote $V_i=V\cap\lin(U_i)$ and consider facets $Z(V_i)$ of the zonotope $Z(V),$ these faces are exactly facets because of corollary \ref{faces2}. It is clear that $Z(V_0)=P$ and $Z(V_n)=Q$ and moreover facets $Z(V_i)$ and $Z(V_{i+1})$ has a common ridge because the dimension of intersection of $V_i$ and $V_{i+1}$ is equal to $d-2.$ So we have constructed a desired belt path on $Z(V)$ and proved the theorem.
\end{proof}

\begin{notat}
Denote as $\xi(d)$ the maximal belt diameter of $d$-dimensional conjugated space-filling zonotope.
\end{notat}

\begin{cor}\label{maxdiam}
The belt diameter of {\bf any} $d$-dimensional space-filling zonotope is not greater than ${\displaystyle \max_{2\leq i\leq d}\xi(i)}.$
\end{cor}

\begin{cor}
In order to find a $d$-dimensional space-filling zonotope with the maximal possible belt diameter it is sufficient to examine all space-filling conjugated zonotopes od dimension at most $d.$ Moreover it is sufficient to analyze conjugated sets $E$ and $F$ satisfying statement of the theorem $\ref{goodbasis}$, i.e. set of vectors $E\cup F$ can be represented as $d\times (2d-2)$-matrix $$E\cup F=\left(\begin{array}{c|c}
E_{d-1}&A\\\hline
0 \ldots 0& 1 \ldots 1
\end{array}
\right)$$
with $0/1$-matrix $A.$
\end{cor}

\begin{thm}\label{diameter}
Belt diameter of any $d$-dimensional space-filling zonotope $P$ is not greater than $\lceil\log_2d\rceil.$
\end{thm}

\begin{proof}
We will use an induction with obvious base $d=1$ and $d=2.$ Assume that we already shown that for $k<d$ belt diameter of $k$-dimensional space-filling zonotope does not exceed $\lceil\log_2k\rceil;$ we will show this for $d$-dimensional space-filling zonotopes ($d\geq 3$). Because of theorem \ref{conjugate} it is enough to show that for any $d$-dimensional space-filling zonotope $Z(E\cup F)$ belt distance between facets $P_E$ and $P_F$ correspondent to sets $E$ and $F$ is at most $\lceil\log_2d\rceil$ (because the function $\lceil\log_2x\rceil$ is non-decreasing on the set $[1,+\infty)$).

Apply the theorem \ref{goodbasis} to zonotope $Z(E\cup F).$ Consider a face of $Z(E\cup F)$ correspondent to supported hyperplane $\pi$ parallel to hyperplane $x_{d-1}=0.$ This facet contains vectors $\ib e_1,\ldots,\ib e_{d-2}$ and at least one half (i.e. $\lceil\frac{d-1}{2}\rceil$) of vectors from the set $F$, because at least one half of entries of last row of the matrix $A$ from the theorem \ref{goodbasis} are zeros. So the hyperplane $\pi$ determines a facet $P_\pi$ adjacent to the face $P_E$ by a ridge determined by the vector set $E\setminus\{\ib e_{d-1}\}.$ Moreover considered facet $P_\pi$ has at least $\lceil\frac{d-1}{2}\rceil$ common vectors with the facet $P_F;$ denote the set of joint vectors as $F_\pi.$

Project the zonotope $Z(E \cup F)$ along the face determined by vector set $F_\pi.$ The resulted zonotope is a parallelohedron and has dimension at most $d-\lceil\frac{d-1}{2}\rceil=\lceil\frac{d}{2}\rceil.$ The belt diameter of projection is not greater than $\lceil\log_2\lceil\frac{d}{2}\rceil\rceil=\lceil\log_2d\rceil-1.$ Also by lemma \ref{projectedpath} the belt distance between facets $P_F$ and $P_\pi$ in $Z(E\cup F)$ does not exceed the belt diameter of this projection hence belt distance between facets $P_E$ and $P_F$ is not greater than $1+(\lceil\log_2d\rceil-1),$ q.e.d.
\end{proof}

\begin{thm}\label{5dim}Belt diameter of $5$-dimensional space-filling zonotope $P$ does not exceed $2.$
\end{thm}

\begin{proof}
As in the previous theorem it is enough to consider conjugated space-filling zonotopes of dimension at most 5. All space-filling zonotopes of dimension at most 4 has diameter at most 2 because of the theorem \ref{diameter} so we need to show that belt diameter of five-dimensional conjugated space-filling zonotope $Z(E\cup F)$ written in the form of theorem \ref{goodbasis} is equal to 2. This diameter cannot be equal to 1 because facets $P_E$ and $P_F$ determined by vector sets $E$ and $F$ does not have a common ridge.

If there is such $i$ that $i$-th row of the $4\times 4$-matrix $A$ has exactly three zeros then hyperplane $x_i=0$ determines a facet of $Z(E\cup F)$ adjacent by three-dimensional faces with each of facet $P_E$ and $P_F$ because hyperplane $x_i=0$ contains three vectors from each of set $E$ and $F.$

In the other case every row of the matrix $A$ contains exactly two zeros and exactly two ones. So every row of matrix $A$ determines a partition of set of columns into two subsets of two elements. There are four rows so some two of these partitions, say for $i$-th and $j$-th rows, coincides. If $i$-th and $j$-th rows of matrix $A$ are equal then all vectors of the set $F$ lies in the hyperplane $x_i=x_j$ and also two vectors from $E$ (except $\ib e_i$ and $\ib e_j$) lies in this plane. If $i$-th and $j$-th rows are not equal then they can be obtain from each other by switching zeros and ones and all vectors of the set $F$ lies in the hyperplane $x_i+x_j=x_5$ and also there are two vectors of the set $E$ lies in the same hyperplane. In both cases these pair of sets $E$ and $F$ are not conjugated and we have a contradiction.
\end{proof}

\begin{cor}\label{diameter_new}
Belt diameter of any $d$-dimensional space-filling zonotope $P$ is not greater than $\lceil\log_2\frac45d\rceil.$
\end{cor}
\begin{proof}
We can the same arguments as in theorem \ref{diameter}. All we need is to replace the base of induction from $d=1$ and $d=2$ to $d\leq 5.$ This new base is true due to theorems \ref{diameter} and \ref{5dim}.
\end{proof}

\begin{rem}
The theorem \ref{diameter} in dimension five gives a non-exact estimate.
\end{rem}

\begin{thm}
In the case of dimension $6$ there exists a space-filling zonotope with combinatorial diameter $3$ so the estimate of the theorem $\ref{diameter}$  and corollary $\ref{diameter_new}$ is sharp.
\end{thm}

\begin{proof}
Consider a conjugated zonotope $Z=Z(E\cup F)\subset \mathbb{R}^6$ defined by the vector set
$$V=E\cup F=\left(
\begin{array}{rrrrr|rrrrr}
1&0&0&0&0&1&1&0&0&0\\
0&1&0&0&0&0&1&1&0&0\\
0&0&1&0&0&0&0&1&1&0\\
0&0&0&1&0&0&0&0&1&1\\
0&0&0&0&1&1&0&0&0&1\\\hline
0&0&0&0&0&1&1&1&1&1
\end{array}\right).$$
Here the first columns of the matrix $V$ are vectors $\ib e_1,\ib e_2,\ib e_3,\ib e_4,\ib e_5$ and the last five are vectors $\ib f_1,\ib f_2,\ib f_3,\ib f_4,\ib f_5$ respectively. Further we will show that this zonotope has belt diameter greater than 2 and is a parallelohedron, this is sufficient to prove this theorem.

Assume that belt diameter of $Z$ is 2 then there exists a facet $P$ of $Z$ with common ridges with  both facets $P_E$ and $P_F$ determined by sets $E$ and $F$. Hence there are at least four vectors from $E$ (except $\ib e_i, i<6$) and at least four vectors from $F$ parallel to $P.$ And this is impossible because in that case $i$-th coordinates of four vectors from $F$ must be equal and $i<6.$ So belt diameter of $Z$ is at least 3.

In order to show that $Z$ is a parallelohedron we will validate the lemma \ref{2or3} for every four-dimensional subset of the vector set $E\cup F.$ Denote the sixth vector of initial basis as $\ib g.$ Note that after cyclic transposition of the first five vectors of the initial basis vectors of sets $E$ and $F$ also transposing cyclically. Moreover any automorphism of the left pentagon on the next figure induces the same automorphism of the right pentagon (rotation corresponds to rotation and axial symmetry corresponds to axial symmetry about parallel line) and both these automorphisms together determines a transposition of the vector set $E\cup F$ preserving the polytope $Z.$

\begin{figure}[!ht]
\begin{center}
\includegraphics{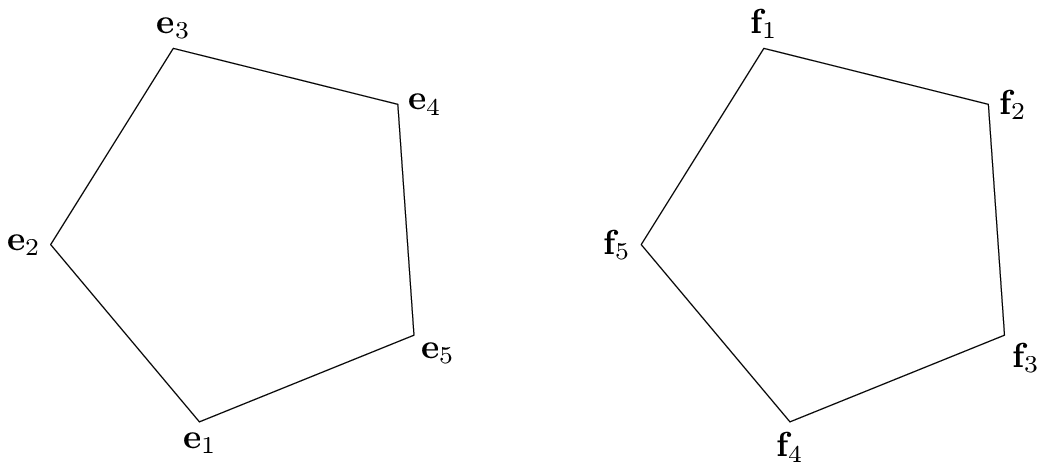}
\caption{The pentagons correspondent to $E$ and $F.$}
\end{center}
\end{figure}

Furthermore vectors of the set $E$ can be written in the basis $\ib f_1,\ib f_2,\ib f_3,\ib f_4,\ib f_5,\ib g$ as 
$$\ib e_i=\frac12(\ib f_i+\ib f_{i+1}+\ib f_{i+3}-\ib f_{i+2}-\ib f_{i+4})-\frac12\ib g.$$
So after replacing the vector $\ib g$ with vector $\frac12(\ib f_1+\ib f_2+\ib f_3+\ib f_4+\ib f_5-\ib g),$ vectors $\ib f_i$ with vectors $-\ib f_{3i}$ and vectors $\ib e_j$ with $\ib e_{3j+3}$ (these transpositions does not change the zonotope $Z(E\cup F)$) we will have a new basis $\{F\cup\ib g\}$ such that the set $F \cup E$ will have the same representation as a matrix in the new basis as the set $E\cup F$ in the old one. (Here we consider indices for sets $E$ and $F$ modulo 5.) We have constructed a map of the left pentagon onto the right pentagon and vice versa and this map does not changing the affine and the combinatorial structure of the zonotope $Z.$ This map allows us to consider only four-dimensional subsets of $E\cup F$ with at least half of vectors from $E.$ Moreover if considered four-dimensional subset contains two vectors from $E$ and two vectors from $F$ then we can examine only the cases with vectors from $E$ which are neighbor on the left pentagon or vectors from $E$ are not neighbor and vector from $F$ are neighbor on the right pentagon (because images of non-neighbors on the right pentagon are neighbors on the left one).

Let $G$ denotes considered four-dimensional subset of $E\cup F$ and $G$ contains at least two vectors from $E.$ We need to examine the following cases:

\begin{enumerate}\renewcommand{\theenumi}{\roman{enumi}}
\item The set $G$ contains four vectors from the set $E,$ without loss of generality (by the first remark about automorphisms of pentagons) we can assume that these vectors are $\ib e_1,\ib e_2,\ib e_3,\ib e_4.$

    We will project remaining vectors of the set $E\cup F$ on the plane generated by vectors $\ib e_5$ and $\ib g.$ The vector $\ib e_5$ projects on itself; vectors $\ib f_1$ and $\ib f_5$ projects on vector $\ib e_5+\ib g;$ vectors $\ib f_2,\ib f_3,\ib f_4$ projects on vector $\ib g.$ Hence all projections will give us vectors of three directions as desired.
\item\label{2} The set $G$ contains three vectors from $E$ and these three vectors are consecutive on the left pentagon. We can assume that $G$ contains vectors $\ib e_1,\ib e_2$ and $\ib e_3$ from the set $E.$ For the fourth vector from $G$ there are several cases:
    \begin{enumerate}
    \item This vector is the vector $\ib f_1$ (or the vector $\ib f_4$; these cases are analogous because vectors $\ib f_1$ and $\ib f_4$ on the right pentagon are symmetric about line of symmetry of the set of vertices $\ib e_1,\ib e_2,\ib e_3$ on the left pentagon).

        After projection of remaining vectors from $E\cup F$ on the plane generated by vectors $\ib e_4$ and $\ib e_5$ we will have the following picture. Vectors $\ib e_4$ and $\ib f_5$ will be projected into vector $\ib e_4;$ vector $\ib e_5$ will be projected into itself; vectors $\ib f_2$ and $\ib f_3$ will be projected into vector $-\ib e_5;$ vector $\ib f_4$ will be projected into vector $\ib e_4-\ib e_5.$ So we will obtain vectors of three directions as desired.

    \item The set $G$ contains the vector $\ib f_2$ (or the vector $\ib f_3$).

        After projection of remaining vectors from $E\cup F$ on the plane generated by vectors $\ib e_4$ and $\ib e_5$ we will have the following picture. Vectors $\ib e_4$ and $\ib f_4$ will be projected into vector $\ib e_4;$ vectors $\ib e_5$ and $\ib f_1$ will be projected into vector $\ib e_5$; vector $\ib f_3$ projects on zero-vectors because it lies in the four-dimensional subspace generated by vectors from $G$ ($\ib f_3=\ib f_2-\ib e_1+\ib e_3$); vector $\ib f_5$ projects on vector $\ib e_4+\ib e_5.$ So we will obtain vectors of three directions as desired.

    \item The set $G$ contains the vector $\ib f_5.$

        After projection of remaining vectors from $E\cup F$ on the plane generated by vectors $\ib e_4$ and $\ib e_5$ we will have the following picture. Vector $\ib e_4$ projects on itself; vector $\ib e_5$ projects on itself; vector $\ib f_1$ projects on vector $-\ib e_4;$ vectors $\ib f_2$ and $\ib f_3$ projects on vector $-\ib e_4-\ib e_5;$ vector $\ib f_4$ projects on vector $-\ib e_5.$ So we will obtain vectors of three directions as desired.
    \end{enumerate}

\item\label{3} The set $G$ contains three vectors from the set $E$ and these vectors are not consecutive on the left pentagon. We can assume that $G$ contains vectors $\ib e_1,\ib e_2$ and $\ib e_4$ from $E.$ For the fourth vector from the set $G$ we have several cases.
    \begin{enumerate}
    \item The set $G$ contains the vector $\ib f_1$ (or the vector $\ib f_3$).

        After projection of remaining vectors from $E\cup F$ on the plane generated by vectors $\ib e_3$ and $\ib e_5$ we will have the following picture. Vector $\ib e_3$ projects on itself; vector $\ib e_5$ projects on itself; vector $\ib f_2$ projects on vector $-\ib e_5;$ vectors $\ib f_3$ and $\ib f_4$ projects on vector $\ib e_3-\ib e_5;$ vector $\ib f_5$ projects on zero-vector ($\ib f_5=\ib f_1-\ib e_1+\ib e_4$). So we will obtain vectors of three directions as desired.

    \item The set $G$ contains the vector $\ib f_2.$

        After projection of remaining vectors from $E\cup F$ on the plane generated by vectors $\ib e_3$ and $\ib e_5$ we will have the following picture. Vectors $\ib e_3,\ib f_3,\ib f_4$ projects on vector $\ib e_3;$ vectors $\ib e_5,\ib f_1,\ib f_5$ projects on vector $\ib e_5.$ So we will obtain vectors of two directions as desired.

    \item The set $G$ contains vector $\ib f_4$ (or vector $\ib f_5$).

        In this case the vector $\ib f_3$ also lies in the four-dimensional face generated by the set $G$ because $\ib f_3=\ib f_4-\ib e_4+\ib e_2$ and the case with vectors $\ib e_1,\ib e_2,\ib e_4,\ib f_3$ already considered in the \ref{3}(a).
    \end{enumerate}
\item The set $G$ contains two vectors from $E$ and two vectors from $F$ and vectors from $E$ are neighbor in the left pentagon. Without loss of generality we can assume the $G$ contains vectors $\ib e_1$ and $\ib e_2.$ There are several cases for vector from $F$ contained in $G.$
    \begin{enumerate}
    \item The set $G$ contains vectors $\ib f_1$ and $\ib f_2$ (or vectors $\ib f_2$ and $\ib f_3$).

        In this case the considered four-dimensional face also contains the vector $\ib e_5=\ib f_1-\ib f_2+\ib e_2$ and the case with three consecutive vectors on the left pentagon already done in the \ref{2}.

    \item The set $G$ contains vectors $\ib f_1$ and $\ib f_3.$

        After projection of remaining vectors from $E\cup F$ on the plane generated by vectors $\ib e_3$ and $\ib e_4$ we will have the following picture. Vectors $\ib e_3$ and $\ib e_5$ projects on vector $\ib e_3;$ vectors $\ib e_4,\ib f_4$ and $\ib f_5$ projects on vector $\ib e_4;$ vector $\ib f_2$ projects on vector $-\ib e_3.$ So we will obtain vectors of two directions as desired.

    \item The set $G$ contains vectors $\ib f_1$ and $\ib f_4$ (or vectors $\ib f_3$ and $\ib f_5$).

        After projection of remaining vectors from $E\cup F$ on the plane generated by vectors $\ib e_3$ and $\ib e_4$ we will have the following picture. Vector $\ib e_3$ projects on itself; vectors $\ib e_4$ and $\ib f_5$ projects on vector $\ib e_4$; vector $\ib e_5$ projects on vector $\ib e_3+\ib e_4$; vector $\ib f_2$ projects on vector $-\ib e_3-\ib e_4;$ vector $\ib f_3$ projects on vector $-\ib e_4.$ So we will obtain vectors of three directions as desired.

    \item The set $G$ contains vectors $\ib f_1$ and $\ib f_5$ (or vectors $\ib f_3$ and $\ib f_4$).

        In this case the four-dimensional face also contains the vector $\ib e_4=\ib f_5-\ib f_1+\ib e_1$ and this already done in \ref{3}.

    \item The set $G$ contains vectors $\ib f_2$ and $\ib f_4$ (or $\ib f_2$ and $\ib f_5$)

        After projection of remaining vectors from $E\cup F$ on the plane generated by vectors $\ib e_3$ and $\ib e_5$ we will have the following picture. Vectors $\ib e_3$ and $\ib f_3$ projects on vector $\ib e_3;$ vector $\ib e_4$ projects on vector $-\ib e_3;$ vectors $\ib e_5$ and $\ib f_1$ projects on vector $\ib e_5;$ vector $\ib f_5$ projects on vector $\ib e_5-\ib e_3.$ So we will obtain vectors of three directions as desired.

    \item The set $G$ contains vectors $\ib f_4$ and $\ib f_5.$

        After projection of remaining vectors from $E\cup F$ on the plane generated by vectors $\ib e_3$ and $\ib e_4$ we will have the following picture. Vectors $\ib e_3$ and $\ib e_5$ projects on vector $\ib e_3;$ vectors $\ib e_4$ projects on itself; vectors $\ib f_1$ and $\ib f_3$ projects on vector $-\ib e_4;$ vector $\ib f_2$ projects on vector $-\ib e_3-\ib e_4.$ So we will obtain vectors of three directions as desired.

    \end{enumerate}

\item The set $G$ contains two vectors from $E$ and two vectors from $F$ and vectors from $E$ are not neighbors on the left pentagon. We need to consider only the case with consecutive vectors from $F$ on the right pentagon because of our second note about automorphisms of pentagons non-neighbors on the right pentagons corresponds to neighbors on the left one and this case we have already considered. Also without loss of generality we can assume that the set $G$ contains vectors $\ib e_1$ and $\ib e_3$ from the set $E.$ For vectors from $F$ contained in $G$ we have the following cases.

    \begin{enumerate}
    \item The set $G$ contains vectors $\ib f_1$ and $\ib f_2$ (or vectors $\ib f_3$ and $\ib f_4$).

        Then our four-dimensional face also contains vector $\ib f_3=\ib f_2-\ib e_1+\ib e_3,$ so it contains three vectors from $F$ and the case if four-dimensional case contains 3 vectors from one of sets already done in \ref{2} and \ref{3}.

    \item The set $G$ contains vectors $\ib f_2$ and $\ib f_3.$

        Then the set $G$ also must contain at least one more vector from the set $E\cup F$ because vectors $\ib e_1,\ib e_3,\ib f_2,\ib f_3$ are linearly dependent ($\ib f_3=\ib f_2-\ib e_1+\ib e_3$) and does not determine four-dimensional face of the zonotope $Z(E\cup F).$

    \item The set $G$ contains vectors $\ib f_4$ and $\ib f_5$ (or vectors $\ib f_5$ and $\ib f_1$).

        Then this four-dimensional face also contains vector $\ib e_5=\ib f_5-\ib f_4+\ib e_3$ and this case already done in \ref{3}.
    \end{enumerate}
\end{enumerate}

So we considered all possible cases and the theorem has been proved.
\end{proof}

\section{Acknowledgements}

I would like to thank Nikolai Dolbilin for helpful discussion and stating this problem. Also I wish to thank Andrey Gavrilyuk and Alexander Magazinov for discussion and helpful remarks and references.

\end{document}